\documentclass{amsart}
\usepackage{amsmath}
\usepackage{amsfonts}
\usepackage{graphics}
\usepackage{epsfig}
\usepackage{amssymb}
\usepackage{amscd}
\usepackage[all]{xy}
\usepackage{latexsym}
\usepackage{graphicx}
\usepackage{multirow}
\usepackage{geometry}
\usepackage{color}
\usepackage[usenames,dvipsnames]{pstricks}
\usepackage{epsfig}
\usepackage{pst-grad} 
\usepackage{pst-plot}

\newtheorem{thm}{Theorem}[section]

\newtheorem{lem}[thm]{Lemma}
\newtheorem{prop}[thm]{Proposition}

\newtheorem{quest}[thm]{Question}

\theoremstyle{definition}
\newtheorem{Def}[thm]{Definition}
\newtheorem{rem}[thm]{Remark}
\newtheorem*{ack}{Acknowledgement}
\newtheorem{ex}[thm]{Example}

\numberwithin{equation}{section}
\numberwithin{figure}{section}

\def\rchi{{\hbox{\raise1.5pt\hbox{$\chi$}}}}

\def\a{\alpha}
\def\b{\beta}

\def\gam{\gamma}
\def\Gam{\Gamma}

\newcommand{\bea}{\begin{eqnarray}}
\newcommand{\eea}{\end{eqnarray}}
\newcommand{\be}{\begin{equation}}
\newcommand{\ee}{\end{equation}}

\newcommand{\Mbar}{{\overline{\mathcal{M}}}}

\newcommand{\bP}{{\mathbb{P}}}
\newcommand{\bC}{{\mathbb{C}}}

\newcommand{\bQ}{{\mathbb{Q}}}

\newcommand{\bZ}{{\mathbb{Z}}}

\newcommand{\cH}{{\mathcal{H}}}

\newcommand{\la}{{\langle}}
\newcommand{\ra}{{\rangle}}
\newcommand{\half}{{\frac{1}{2}}}

\newcommand{\rar}{\rightarrow}
\newcommand{\lrar}{\longrightarrow}

\begin{document}
\large
\setcounter{section}{0}

\allowdisplaybreaks

\title[Mirror of orbifold Hurwitz 
numbers]
{Mirror curve of orbifold Hurwitz numbers}

\author[O.\ Dumitrescu]{Olivia Dumitrescu}
\address{O.~Dumitrescu: 
Department of Mathematics\\
Central Michigan University\\
Mount Pleasant, MI 48859}
\email{dumit1om@cmich.edu}
\address{and Simion Stoilow Institute of Mathematics\\
Romanian Academy\\
21 Calea Grivitei Street\\
010702 Bucharest, Romania}

\author[M.\ Mulase]{Motohico Mulase}
\address{M.~Mulase:
Department of Mathematics\\
University of California\\
Davis, CA 95616--8633}
\email{mulase@math.ucdavis.edu}
\address{and Kavli Institute for Physics and Mathematics of the 
Universe\\
The University of Tokyo\\
Kashiwa, Japan}

\thanks{The first author is supported by NSF grant DMS1802082.}

\begin{abstract}
 Edge-contraction
operations  form an effective tool in
various graph enumeration problems, such
as counting Grothendieck's dessins d'enfants
 and simple and double Hurwitz numbers. These
counting problems can be solved by
 a mechanism known as 
  topological recursion, which is a
   mirror B-model corresponding to 
 these counting problems.
 We show that for the case of
 orbifold Hurwitz numbers, the mirror objects, i.e.,
 the spectral curve and the differential forms on it,
 are constructed solely from
 the  edge-contraction operations of the counting
 problem in genus $0$ and one \emph{marked}
 point. This forms a parallelism with 
 Gromov-Witten theory, where genus $0$ 
 Gromov-Witten invariants correspond to mirror
 B-model holomorphic geometry. 
\end{abstract}



\maketitle

\tableofcontents

\section{Introduction}
\label{sect:intro}

The purpose
 of the present paper is to identify
 the mirror B-model objects that enable us to solve
 certain graph enumeration problems. We consider
 simple and orbifold Hurwitz numbers,
 by giving a graph enumeration formulation for
 these numbers. We then show that the mirror
 of these counting problems are constructed
 from the \emph{edge-contraction operations}
 of \cite{OM3} applied to 
 orbifold 
 Hurwitz numbers for the case of
 \emph{genus $0$ and one-marked point}.

Edge-contraction operations 
  provide an effective method
for  graph enumeration problems. 
It has been noted in 
\cite{DMSS} that the Laplace transform
of  edge-contraction operations on many
counting problems corresponds
to the topological recursion of
\cite{EO2007}. In
 this paper, we  examine the 
construction of  mirror B-models corresponding
to  the simple
and orbifold Hurwitz numbers. 
In general,
enumerative geometry problems, 
such as computation of Gromov-Witten type
invariants, are often solved by 
studying a corresponding problem
on the  \emph{mirror dual} side. 
The effectiveness of the mirror method  relies
on complex analysis and holomorphic 
geometry technique that is available on 
the mirror B-model side.
The  question we consider
in this paper is the following:

\begin{quest}
How do we find the mirror of a given
enumerative problem?
\end{quest}

We give an answer to this question
for a class of graph enumeration problems
that are equivalent to counting 
orbifold Hurwitz numbers. 
The key is  the  edge-contraction operations. 
The base case, or the case for the ``moduli space''
$\Mbar_{0,1}$, of the edge contraction in the
counting problem identifies the mirror dual
object, and a universal mechanism
of complex analysis, known as the
\textbf{topological recursion} of \cite{EO2007},
solves the B-model side of the counting
problem. The solution is a collection of
generating functions of the original counting problem
for all genera.

Bouchard and Mari\~no \cite{BM} conjectured
that generating functions for simple Hurwitz 
numbers could be calculated by the topological 
recursion
of \cite{EO2007}, based on the spectral curve
identified as the \textbf{Lambert curve}
 \be
 \label{Lambert}
 x = y e^{-y}.
 \ee
 Here, the notion of spectral curve
is the mirror dual object for the counting problem.
They arrived at the mirror dual by a consideration
of mirror symmetry of 
\emph{open} Gromov-Witten 
invariants of toric Calabi-Yau threefolds
\cite{BKMP1}. The mirror geometry of a toric
Calabi-Yau threefold is completely determined by
a plane algebraic curve known as the
\emph{mirror curve}. The Lambert curve
\eqref{Lambert} appears as the infinite framing number
limit of the mirror curve of $\bC^3$.
The Hurwitz number 
conjecture of \cite{BM} 
was then solved in a series of papers by
one of the authors \cite{EMS,MZ}, using
the Lambert curve as  a \emph{given} input. 
Since conjecture is true, the Lambert curve
\eqref{Lambert} \emph{should be} the mirror
B-model for Hurwitz numbers. But why? 
In \cite{EMS,MZ}, we did not attempt to give any
explanation.

The emphasis of
our current paper is to prove  that the mirror dual object
is simply a consequence of the $\Mbar_{0,1}$ case of 
the edge-contraction operation on the original 
counting problem.
The situation is similar to 
several cases of Gromov-Witten theory,
where the mirror is constructed by the genus $0$
Gromov-Witten invariants themselves.

To illustrate the idea, let us consider the number $T_d$
of
connected  \emph{trees} consisting of \emph{labeled}
$d$ nodes (or vertices). The initial condition is 
$T_1 = 1$. The numbers satisfy a recursion relation
\be
\label{tree}
(d-1) T_d = \half  \sum_{\substack{a+b=d\\
a,b\ge 1}} ab \binom{d}{a}T_aT_b.
\ee
A tree of $d$ nodes has $d-1$ edges. The left-hand
side counts how many ways we can eliminate
an edge. When an edge is eliminated, the tree
breaks down into two disjoint pieces, one consisting
of $a$ labeled nodes, and the other $b=d-a$ 
labeled nodes.
The original tree is restored by connecting
one of the $a$ nodes on one side to
 one of the $b$ nodes on the other side. The
 equivalence of  counting in this elimination
 process  gives \eqref{tree}.
 From the initial value, the recursion
 formula generates the tree sequence
 $1,1,3,16,125,1296,\dots$. We note, however,
 that \eqref{tree} does not directly give a
 closed formula for $T_d$. To find one,
 we introduce a generating function, or
 a \textbf{spectral curve}
 \be
 \label{tree y}
 y = y(x)  := \sum_{d=1}^\infty \frac{T_d}{(d-1)!}
 x^d.
 \ee
 In terms of the generating 
 function,  \eqref{tree} becomes equivalent to 
 \be
 \label{tree diff}
 \left(x^2\circ \frac{d}{dx}\circ \frac{1}{x}
 \right) y = \half x\frac{d}{dx} y^2
 \Longleftrightarrow
 \frac{dx}{dy} = \frac{x(1-y)}{y}.
 \ee
 The initial condition is $y(0) = 0$ and $y'(0)=1$,
 which allows us to 
 solve the differential equation uniquely. Lo and behold,
 the solution is exactly
 \eqref{Lambert}.

  To find the formula for $T_d$, we need the 
 \emph{Lagrange Inversion Formula}.
 Suppose that $f(y)$ is a holomorphic function  
  defined near $y=0$, and 
 that $f(0)\ne 0$. Then the inverse
 function of
 $
 x = \frac{y}{f(y)}
 $
 near $x=0$ is given by
 \be
 \label{LIF}
 y = \sum_{k=1}^\infty \left.\left(\frac{d}{dy}
 \right)^{k-1}
 \big(f(y)^k\big)\right|_{y=0}\frac{x^k}{k!}.
 \ee
The proof is elementary and
requires only Cauchy's integration 
formula. Since $f(y)=e^y$ in our case, we immediately
obtain Cayley's formula $T_d = d^{d-2}$.

The point we wish to make here is that the real
problem behind the scene is not tree-counting, but
 \emph{simple Hurwitz numbers}.
This relation is understood by the correspondence
between trees and ramified coverings 
of $\bP^1$ by $\bP^1$ of degree $d$ that are
simply ramified except for one total ramification
point.
When we
look at
the dual graph of a tree,    elimination of an edge
becomes  contracting an edge, and 
this operation 
precisely gives a \emph{degeneration formula}
for counting problems on $\Mbar_{g,n}$. 
The base case for the counting problem
is $(g,n) = (0,1)$, and the recursion \eqref{tree} is
 the result of the edge-contraction operation 
for simple Hurwitz numbers 
associated with $\Mbar_{0,1}$. In this sense,
the Lambert curve \eqref{Lambert}
is the \emph{mirror dual}
 of  simple Hurwitz numbers.

The paper is organized as follows. 
In Section~\ref{sect:Hurwitz}, we 
present combinatorial graph enumeration
problems, and show that they are  equivalent
to  counting of
simple and orbifold Hurwitz numbers.
In Section~\ref{sect:spectral},
the spectral curves 
of the topological recursion for simple 
and orbifold Hurwitz numbers 
(the mirror objects to the counting problems)
are constructed
from the edge-contraction formulas for 
$(g,n) = (0,1)$ invariants.

\section{Orbifold Hurwitz numbers
as graph enumeration}
\label{sect:Hurwitz}

Mirror symmetry provides an effective tool
for counting problems of Gromov-Witten type
invariants. The question is how we construct 
the mirror, given a counting problem.
Although there is so far no general formalism, 
we present a systematic procedure for 
computing orbifold Hurwitz numbers in this 
 paper. The key observation is that
the edge-contraction operations for $(g,n)=(0,1)$
identify the mirror object.

The topological recursion for 
simple and orbifold Hurwitz numbers 
are derived as the Laplace transform of
the cut-and-join equation
\cite{BHLM, EMS, MZ}, where the spectral
curves are identified by the consideration of
mirror symmetry of toric Calabi-Yau
orbifolds \cite{BHLM,BM, FLZ,FLZ2}. In this section we give a
purely combinatorial graph enumeration 
problem that is equivalent to counting
orbifold Hurwitz numbers.
We then show in the next section
that the edge-contraction 
formula restricted to the $(g,n)=(0,1)$
case determines the spectral curve
and the differential forms $W_{0,1}$ and
$W_{0,2}$ of \cite{BHLM}. These
quantities form the mirror objects for the
orbifold Hurwitz numbers.

\subsection{Cell graphs}

To avoid unnecessary confusion, we use
the terminology  \emph{cell graphs}
 in this article, instead of more common
ribbon graphs. Ribbon graphs naturally
appear for encoding complex structures of
a topological surface 
(see for example, \cite{K1992, MP1998}).
Our purpose of using ribbon graphs are
for degeneration of stable curves, and we label
vertices, instead of \emph{faces}, of a ribbon 
graph.

\begin{Def}[Cell graphs]
A connected \textbf{cell graph} 
of topological type $(g,n)$ is the
$1$-skeleton of a cell-decomposition of a
connected
closed oriented surface of genus $g$ with 
$n$ labeled $0$-cells. We call a $0$-cell a 
\emph{vertex}, a $1$-cell an \emph{edge}, 
and a $2$-cell a \emph{face}, of the cell graph.
We denote by $\Gam_{g,n}$ the set of 
connected cell graphs of type $(g,n)$.
Each edge consists of two \textbf{half-edges}
connected at the midpoint of the edge.
\end{Def}

\begin{rem}
\begin{itemize}
\item
The \emph{dual} of a cell graph is
a ribbon graph, or Grothendieck's 
dessin d'enfant. We note that we label vertices
of a cell graph, which corresponds to 
face labeling of a ribbon graph.
Ribbon graphs are also called by different names,
such as
embedded graphs and maps.

\item We identify two cell graphs if there is a 
homeomorphism of the surfaces that brings 
one cell-decomposition to the other, 
keeping the labeling of $0$-cells. The only 
possible automorphisms of a cell graph
come from cyclic rotations of half-edges
at each vertex. 
\end{itemize}
\end{rem}

\begin{Def}[Directed cell graph]
A \textbf{directed cell graph} is a cell graph
for which an arrow is assigned to each edge.
An arrow is the same as an ordering of the 
two half-edges forming an edge.
The set of directed cell graphs of type
$(g,n)$ is denoted by $\vec{\Gam}_{g,n}$.
\end{Def}

\begin{rem}
A directed cell graph is a \emph{quiver}. Since
our graph is drawn on an oriented surface, 
a directed cell graph carries more information than
its underlying quiver structure. The tail vertex
of an arrowed edge is called the \emph{source},
and the head of the arrow the \emph{target}, in the
quiver language.
\end{rem}

An effective tool in graph enumeration is
edge-contraction operations. Often edge contraction
leads to an inductive formula for counting problems
of graphs.

\begin{Def}[Edge-contraction operations]
\label{def:ECO}
There are two types of 
 \textbf{edge-contraction operations} 
 applied to cell graphs. 
 \begin{itemize}

\item \textbf{ECO 1}: Suppose there is a
directed edge $\vec{E}=\overset{\lrar}{p_ip_i}$
in a cell graph $\gam\in \vec{\Gam}_{g,n}$,
connecting the tail vertex $p_i$ and the head
vertex $p_j$.
We  \emph{contract} $\vec{E}$ in $\gam$,
and put the two vertices $p_i$ and $p_j$ together.
We use $i$ for the label of this new vertex, and 
call it again $p_i$. 
Then we have a new cell graph
 $\gam'\in \vec{\Gam}_{g,n-1}$ with one less vertices.
 In this process, the topology of the surface on
 which $\gam$ is drawn does not change. Thus
 genus $g$ of the graph stays the same.

\begin{figure}[htb]
\includegraphics[height=1in]{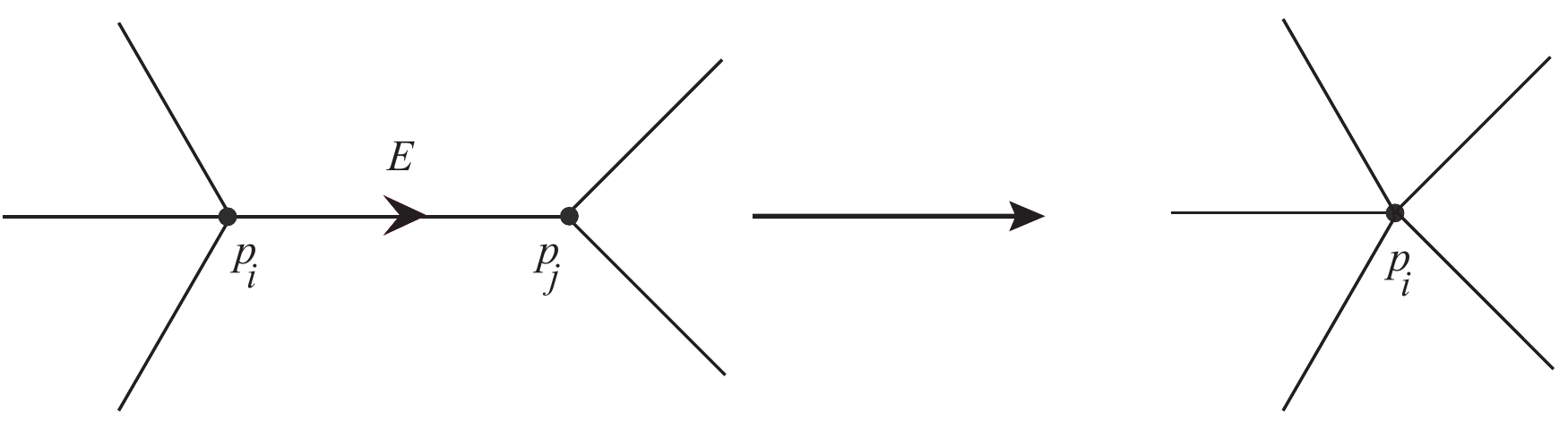}
\caption{Edge-contraction operation ECO 1.
The edge bounded by two vertices $p_i$ and $p_j$
is contracted to a single vertex $p_i$.
}
\label{fig:ECO 1}
\end{figure}

\item We use the notation $\vec{E}$ for the 
 edge-contraction operation 
 \be\label{ECO1}
 \vec{E}:\vec{\Gam}_{g,n}\owns \gam\longmapsto \gam'\in 
 \vec{\Gam}_{g,n-1}.
 \ee

\item \textbf{ECO 2}: Suppose there is a directed loop
$\vec{L}$ in $\gam\in\vec{\Gam}_{g,n}$ at the $i$-th vertex $p_i$.
Since a loop in the $1$-skeleton of
a cell decomposition is 
a topological cycle on the surface, its contraction 
inevitably changes the topology of the surface. 
First we look at the half-edges incident to 
vertex $p_i$. Locally around $p_i$ on the 
surface, the directed loop $\vec{L}$ 
separates the neighborhood of
$p_i$
into two pieces. Accordingly, we put 
the incident half-edges into  
two groups. We then break the vertex $p_i$
into two vertices, $p_{i_1}$ and $p_{i_2}$, so that 
one group of half-edges are incident to $p_{i_1}$,
and the other group to $p_{i_2}$. 
The order of two vertices 
is determined by placing the loop $\vec{L}$
\emph{upward} near at vertex $p_i$. 
Then we name the new vertex on its left by $p_{i_1}$,
and on its right by $p_{i_2}$.

Let $\gam'$ denote the possibly 
disconnected graph obtained by 
contracting $\vec{L}$ and separating the vertex
to two distinct vertices labeled by $i_1$ and $i_2$.

\begin{figure}[htb]
\includegraphics[height=1.1in]{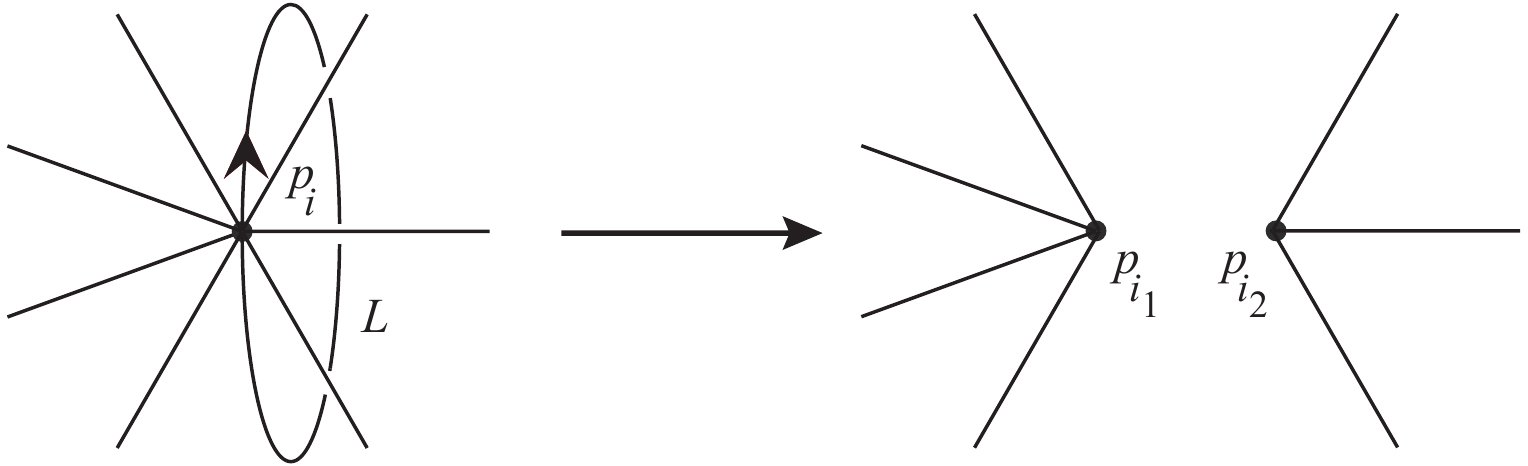}
\caption{Edge-contraction operation 
ECO 2. The contracted edge is a loop $\vec{L}$ of 
a cell graph. Place the loop 
so that it is upward near at $p_i$
to which 
$\vec{L}$ is attached. The vertex $p_i$  is 
then broken into two vertices,
$p_{i_1}$ on the left, and $p_{i_2}$ on 
the right. Half-edges 
incident to $p_i$ are separated into two groups,
belonging to two sides of the loop near $p_i$.
}
\label{fig:ECO2}
\end{figure}

\item
If $\gam'$ is connected, then it is in $\vec{\Gam}_{g-1,n+1}$.
The loop $\vec{L}$  is a \textit{loop of handle}.
We use the same notation $\vec{L}$ to indicate
the edge-contraction operation
\be\label{ECO2-1}
\vec{L}:\vec{\Gam}_{g,n}\owns \gam\longmapsto
\gam'\in \vec{\Gam}_{g-1,n+1}.
\ee

\item
If $\gam'$ is disconnected, then write
$\gam'=(\gam_1,\gam_2)\in \vec{\Gam}_{g_1,|I|+1}
\times \vec{\Gam}_{g_2,|J|+1}$, where 
\be\label{disconnected}
\begin{cases}
g=g_1+g_2\\
I\sqcup J = \{1,\dots,\widehat{i},\dots,n\}
\end{cases}.
\ee
The edge-contraction operation is again denoted 
by
\be\label{ECO2-2}
\vec{L}:\vec{\Gam}_{g,n}\owns \gam\longmapsto
(\gam_1,\gam_2)\in \vec{\Gam}_{g_1,|I|+1}\times
\vec{\Gam}_{g_2,|J|+1}.
\ee
In this case we  call $\vec{L}$ a \textit{separating loop}.
Here, vertices labeled by $I$ belong to the connected 
component of genus 
$g_1$, and those labeled by $J$ are on the other 
component of genus $g_2$. 
Let $(I_-,i,I_+)$ (reps. $(J_-,i,J_+)$) be the reordering of 
$I\sqcup \{i\}$ (resp. $J\sqcup \{i\}$)
in the increasing order. 
Although we give labeling $i_1,i_2$ to the
two vertices created by breaking $p_i$,
since they belong to distinct graphs, we can simply
use $i$ for the label of $p_{i_1}\in \gam_1$ and
the same $i$ for $p_{i_2}\in \gam_2$. 
The arrow of $\vec{L}$ translates
into the information of ordering among the 
two vertices $p_{i_1}$ and $p_{i_2}$. 
\end{itemize}
\end{Def}

\begin{rem}
The use of directed cell graphs enables us to 
define edge-contraction operations, keeping track
with vertex labeling. We refer to \cite{OM6}
for the actual motivation for quiver cell graphs. 
Since our main concern is enumeration of graphs,
the extra data of directed edges does not plan any role.
In what follows, we deal with cell graphs without
directed edges. The edge-contraction
operations are defined with a choice of direction, but
the counting formula we derive does not depend of this choice.
\end{rem}

\begin{rem}
Let us define $m(\gam)=2g-2+n$
for a graph $\gam\in \Gamma_{g,n}$. Then every
edge-contraction operation  reduces 
$m(\gam)$ exactly by $1$. 
Indeed, for ECO 1, we have 
$$
m(\gam') = 2g -2
+(n-1) = m(\gam)-1.
$$
The ECO 2 applied to a loop of handle produces
$$
m(\gam') = 2(g-1)-2+(n+1) = m(\gam)-1.
$$ 
For a separating loop, we have
$$
\begin{matrix}
&2g_1-2+|I|+1
\\
{+)}&{2g_2-2+|J|+1}
\\
&\overline{2g_1+2g_2-4+|I|+|J|+2}
&=\; \;2g-2+n-1.
\end{matrix}
$$
\end{rem}

\subsection{$r$-Hurwitz graphs}

We choose and fix a positive integer $r$. 
The decorated graphs we wish to enumerate
are the following.

\begin{Def}[$r$-Hurwitz graph]
An $r$-\textbf{Hurwitz graph} 
$(\gam,D)$ of type $(g,n,d)$ consists of the
following data.
\begin{itemize}

\item $\gam$ is a 
connected cell graph of type 
$(g,n)$, with $n$ labeled vertices.

\item $|D|=d$ is divisible by $r$, and
$\gam$ has $m=d/r$ unlabeled
faces and $s$ unlabeled edges, where
\be
\label{s}
s= 2g-2+\frac{d}{r} + n.
\ee

\item $D$ is a configuration of
$d=rm$ unlabeled dots on the graph subject to 
the following conditions:

\begin{enumerate}
\item The set of $d$ dots are grouped into
$m$ subsets of $r$ dots, each of which
 is equipped with 
a cyclic order.

\item Every face of $\gam$ has 
cyclically ordered $r$ dots.

\item These dots are clustered near 
vertices of the face. At each corner of the
face, say at Vertex $i$, the dots are ordered according
to the cyclic order that is consistent of 
the orientation of the face, which is chosen 
to be counter-clock wise.

\item  Let 
$\mu_i$ denote the total
number  of dots clustered at Vertex $i$. Then
$\mu_i>0$ for every $i=1,\dots,n$.
Thus we have an ordered partition 
\be
\label{ordered partition}
d = \mu_1+\cdots+\mu_n.
\ee

In particular,  the number of vertices ranges 
$0< n\le d$.

\item Suppose an edge $E$ connecting two distinct
vertices, say Vertex $i$ and $j$, bounds the same
face twice. Let $p$ be the midpoint of $E$. 
The polygon representing the face has $E$ twice
on its perimeter,
hence the point $p$ appears also twice. We name
them as $p$ and $p'$. Which one we call $p$ or $p'$
does not matter.
Consider a path on the perimeter of this polygon
starting from $p$ and ending up with 
$p'$ according to the  counter-clock
wise orientation.
 Let $r'$ be the total number of dots 
clustered around vertices of the face, 
counted along the path. Then it satisfies 
\be
\label{dot condition}
0<r'<r.
\ee
 For example,
not all $r$ dots of a face can be clustered 
at a vertex of degree $1$. 
In particular, for the case of $r=1$, 
the graph $\gam$ has no edges bounding
 the same face twice.

\end{enumerate}
\end{itemize}
An \textbf{arrowed} $r$-Hurwitz graph
$(\gam,\vec{D})$  has,
in addition to
to the above data $(\gam,D)$, an arrow 
assigned to  one of the $\mu_i$ dots
from Vertex $i$ for each index $1\le i \le n$.
\end{Def}

The counting problem we wish to study
is the number $\cH_{g,n}^r(\mu_1\dots,\mu_n)$
of arrowed $r$-Hurwitz graphs for a prescribed 
ordered partition \eqref{ordered partition},
counted with the automorphism weight.
The combinatorial data corresponds to 
an object in algebraic geometry. Let us first
identify what the $r$-Hurwitz graphs represent.
We denote by $\bP^1[r]$ the $1$-dimensional
orbifold modeled on $\bP^1$ that has
 one stacky point 
$\left[0 \big/\big(\bZ/(r)\big)\right]$
at $0\in \bP^1$.

\begin{ex} The base case is $\cH_{0,1}^r(r)=1$
(see Figure~\ref{fig:H01}). This counts
the identity morphism $\bP^1[r]
\overset{\sim}{\lrar} \bP^1[r]$.
\end{ex}

\begin{figure}[htb]
\centerline{\epsfig{file=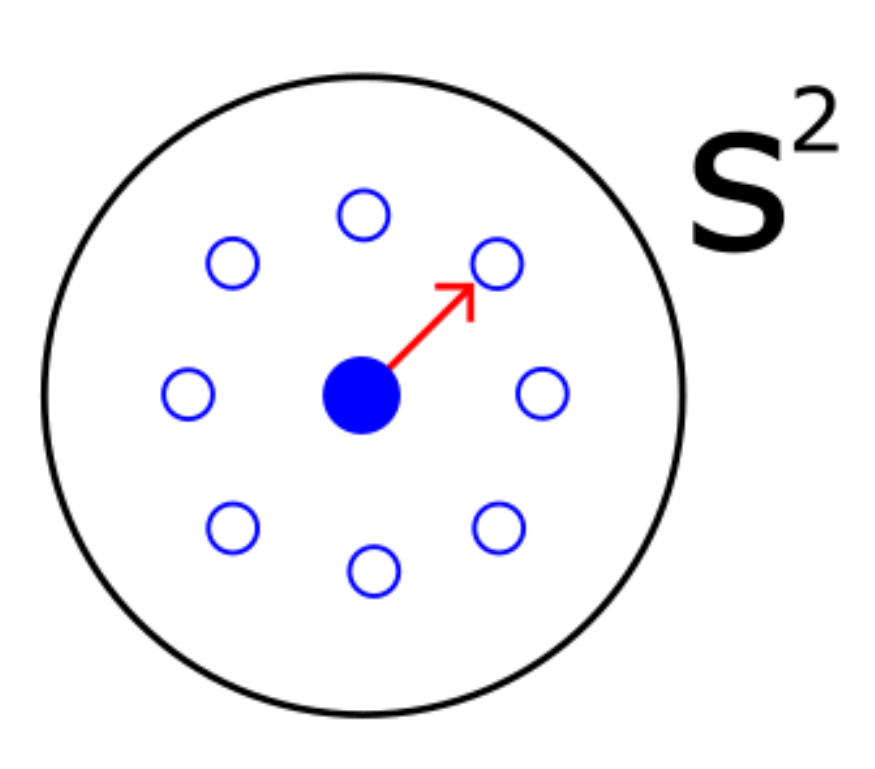, width=0.7in}}
\caption{The graph
has only one vertex and no edges. All $r$ dots
are clustered around this unique vertex, with
an arrow attached to one of them. 
Because of the arrow, there is no automorphism 
of this graph.
}
\label{fig:H01}
\end{figure}

\begin{Def}[Orbifold Hurwitz cover and Orbifold
Hurwitz numbers]
An \emph{orbifold Hurwitz cover} $f:C\lrar \bP^1[r]$ is
a morphism from an orbifold $C$ that is modeled
on a smooth algebraic curve of genus $g$
that has 
\begin{enumerate}
\item
 $m$ stacky points of the same type as
the one on the base curve that are all mapped to
$\left[0 \big/\big(\bZ/(r)\big)\right]\in\bP^1[r]$,
\item 
arbitrary profile $(\mu_1,\dots,\mu_n)$
with $n$ labeled points
over $\infty\in \bP^1[r]$, 
\item and all other
ramification points are simple.
\end{enumerate}
If we replace the target orbifold by $\bP^1$, then
the morphism is a regular map from a
smooth curve of genus $g$ with profile
$(\overset{m}{\overbrace{r,\dots,r})}$
over $0\in\bP^1$, labeled profile
$(\mu_1,\dots,\mu_n)$ over $\infty\in\bP^1$,
and a simple ramification at any other
 ramification point.
The Euler characteristic condition 
\eqref{s} of the
graph $\gam$ gives the number of simple
ramification points of $f$ 
through the Riemann-Hurwitz
formula.
The automorphism weighted count of
the number of the topological types of 
such covers is denoted by 
$H_{g,n}^r (\mu_1,\dots,\mu_n)$.
These numbers are referred to as \emph{orbifold
Hurwitz numbers}. When $r=1$, they
count the usual simple Hurwitz numbers.
\end{Def}

The counting of the topological types is the same
as counting actual orbifold Hurwitz covers
such that all simple ramification points
are mapped to one of  the $s$-th roots of unity
$\xi^1,\dots,\xi^s$, where $\xi = \exp(2\pi i/s)$,
if all simple ramification points of $f$ are
labeled. Indeed, such a labeling is given by
 elements of the cyclic group
$\{\xi^1,\dots,\xi^s\}$ of order $s$.
Let us construct an edge-labeled Hurwitz graph
from an orbifold Hurwitz cover with 
fixed branch points on the target as above.
We first review the case of $r=1$, i.e., 
the simple Hurwitz covers. Our graph is
essentially the same as the dual of the
\emph{branching graph} of \cite{OP}.

\subsection{Construction of $r$-Hurwitz 
graphs}

First we consider the case $r=1$.
Let $f:C\lrar \bP^1$ be a simple 
Hurwitz cover of genus $g$ and degree $d$
with labeled profile
$(\mu_i,\dots,\mu_n)$ over $\infty$,
unramified over $0\in\bP^1$, 
 and simply ramified over
$B=\{\xi^1,\dots,\xi^s\}
\subset \bP^1$, where $\xi = \exp(2\pi i/s)$
and $s=2g-2+d+n$.
We denote by $R = \{p_1,\dots,p_s\}\subset C$ the
labeled simple ramification points of $f$, that is
bijectively mapped to $B$
 by $f:R\lrar B$. We choose a labeling of
 $R$ so that $f(p_\a)= \xi^\a$ for every
 $\a=1,\dots,s$. 
 
 On  $\bP^1$, plot
 $B$ and connect each element $\xi^\a\in B$ with 
  $0$ by a straight line segment. 
 We also 
 connect $0$ and $\infty$ by a straight line
 $z=t \exp(\pi i/s)$, $0\le t\le \infty$.
 Let $*$ denote the configuration of 
 the $s$ line segments.
 The inverse image $f^{-1}(*)$ is a  
 cell graph on $C$, for which $f^{-1}(0)$
 forms the set of vertices. We remove all inverse images 
 $f^{-1}(\overline{0\xi^\a})$ 
 of the line segment $\overline{0\xi^\a}$ from this
 graph, except for the ones that end at one of the
 points $p_\a\in R$.
 Since $p_\a$ is a simple ramification point of
 $f$, the line segment ending at $p_\a$
 extends to another
 vertex, i.e., another
 point in $f^{-1}(0)$. We denote by 
 $\gam^{\vee}$ the graph after this removal
 of line segments. We define the edges
 of the graph to be the connected line
 segments at $p_\a$ for some $\a$. We use 
 $p_\a$ as the label of the edge.
 The graph $\gam^{\vee}$
 has $d$ vertices, $s$ edges,
 and $n$ faces.

 An inverse image of the line $\overline{0\infty}$
 is a  ray starting at a vertex of the graph 
 $\gam^{\vee}$ and ending up with one
 of the points in $f^{-1}(\infty)$, which is
 the center of a face.
 We place a dot on this line
 near at each vertex. The edges of
 $\gam^{\vee}$ incident to a vertex 
 are  cyclically ordered counter-clockwise, 
 following the natural cyclic order of $B$.
 Let $p_\a$ be an edge incident to a vertex,
 and $p_\b$ the next one 
 at the same vertex according to the 
 cyclic order. We denote by $d_{\a\b}$ the number
 of dots in the span of two edges $p_\a$ and $p_\b$,
 which is $0$ if $\a<\b$, and $1$ if $\b<\a$.
 Now we consider 
 the dual graph $\gam$ of 
 $\gam^{\vee}$. It  has $n$ vertices,
 $d$ faces, and $s$ edges still labeled by
 $\{p_1,\dots,p_s\}$. At the angled 
 corner between the two
 adjacent edges labeled by $p_\a$ and $p_\b$
 in this order according to the cyclic order, we
 place $d_{\a\b}$ dots. The 
 data $(\gam,D)$ consisting of the cell 
 graph $\gam$ and the dot configuration $D$ is 
 the Hurwitz graph corresponding to the
 simple Hurwitz cover $f:C\lrar \bP^1$ for $r=1$.

It is obvious that what we obtain is an $r=1$
Hurwitz graph, except for the condition (5)
of the configuration $D$, which requires an
explanation. The dual graph
$\gam^{\vee}$ for $r=1$ is 
the \emph{branching graph} of \cite{OP}. 
Since $|B|=s$ is the number of simple ramification
points, which is also the number of edges of 
$\gam^{\vee}$,
 the branching 
graph cannot have any loops. This is because
two distinct  powers of $\xi$ in the range of
$1,\dots,s$ cannot
be the same.
This fact reflects in
the condition that $\gam$ has no edge
that bounds the same face twice. 
This explains the condition (5) for $r=1$. 

\begin{rem}
If we consider the case $r=1, g=0$ and $n=1$, then
$s=d-1$. Hence the graph $\gam^\vee$
is a connected tree consisting of $d$ 
nodes (vertices) and $d-1$ \emph{labeled} edges.
Except for $d=1,2$, every vertex is uniquely labeled 
 by incident edges. The tree counting of
 Introduction is relevant to Hurwitz numbers in this
 way.
\end{rem}

Now let us consider
  an orbifold 
Hurwitz cover $f:C\lrar \bP^1[r]$
of genus $g$ and degree $d=rm$
with labeled profile
$(\mu_i,\dots,\mu_n)$ over $\infty$, 
$m$ isomorphic  stacky points over
$\left[0 \big/\big(\bZ/(r)\big)\right]\in\bP^1[r]$,
 and simply ramified over
$B=\{\xi^1,\dots,\xi^s\}
\subset \bP^1[r]$, where  $s=2g-2+m+n$.
By $R = \{p_1,\dots,p_s\}\subset C$ we indicate the
labeled simple ramification points of $f$, that is
again bijectively mapped to $B$
 by $f:R\lrar B$. We choose the same labeling of
 $R$ so that $f(p_\a)= \xi^\a$ for every
 $\a=1,\dots,s$. 
 
 On  $\bP^1[r]$, plot
 $B$ and connect each element $\xi^\a\in B$ with 
 the stacky point at $0$ by a straight line segment. 
 We also 
 connect $0$ and $\infty$ by a straight line
 $z=t \exp(\pi i/s)$, $0\le t\le \infty$, as before.
 Let $*$ denote the configuration of 
 the $s$ line segments.
 The inverse image $f^{-1}(*)$ is a  
 cell graph on $C$, for which $f^{-1}(0)$
 forms the set of vertices. We remove all inverse images 
 $f^{-1}(\overline{0\xi^\a})$ 
 of the line segment $\overline{0\xi^\a}$ from this
 graph, except for the ones that end at one of the
 points $p_\a\in R$.
 We denote by 
 $\gam^{\vee}$ the graph after this removal
 of line segments. We define the edges
 of the graph to be the connected line
 segments at $p_\a$ for some $\a$. We use 
 $p_\a$ as the label of the edge.
 The graph $\gam^{\vee}$
 has $m$ vertices, $s$ edges.

 The inverse image of the line $\overline{0\infty}$
 form a set of $r$ rays at each vertex of the graph 
 $\gam^{\vee}$, connecting
 $m$ vertices and 
 $n$ centers $f^{-1}(\infty)$ of faces. 
 We place a dot on each line
 near at each vertex. These dots are cyclically ordered
 according to the orientation of $C$, which we
 choose to be counter-clock wise. The edges of
 $\gam^{\vee}$ incident to a vertex 
 are also cyclically ordered in the same way.
 Let $p_\a$ be an edge incident to this vertex,
 and $p_\b$ the next one according to the 
 cyclic order. We denote by $d_{\a\b}$ the number
 of dots in the span of two edges $p_\a$ and $p_\b$.
 Let $\gam$ denote the dual graph of 
 $\gam^{\vee}$. It now has $n$ vertices,
 $m$ faces, and $s$ edges still labeled by
 $\{p_1,\dots,p_s\}$. At the angled 
 corner between the two
 adjacent edges labeled by $p_\a$ and $p_\b$
 in this order according to the cyclic order, we
 place $d_{\a\b}$ dots, again cyclically ordered
 as on $\gam^{\vee}$. The 
 data $(\gam,D)$ consisting of the cell 
 graph $\gam$ and the dot configuration $D$ is 
 the $r$-Hurwitz graph corresponding to the
 orbifold Hurwitz cover $f:C\lrar \bP^1[r]$.

We note that $\gam^\vee$ can have loops, 
unlike the case of $r=1$. Let us place
$\gam^\vee$ locally on an oriented plane around
a vertex. The plane is locally separated into $r$ sectors
by the $r$ rays $f^{-1}(\overline{0\infty})$
at this vertex. There are $s$ half-edges coming
out of the vertex at each of these $r$ sectors.
A half-edge corresponding to $\xi^\a$ cannot be 
connected to another half-edge corresponding
to $\xi^\b$ in the same sector, by the 
same reason for the case of $r=1$. But it
can be connected
to another half-edge of a different sector 
corresponding again to the same $\xi^\a$. In this case, 
within the loop there are some dots, representing
the rays of $f^{-1}(\overline{0\infty})$ in between
these half-edges. The total number of dots in the
loop cannot be $r$, because then the half-edges
being connected are in the same sector. 
Thus the condition (5) is satisfied.

\begin{ex}
Theorem~\ref{thm:ECF} below shows
that 
$$
\cH_{0,2}^2(3,1) = \frac{9}{2}.
$$
This is the weighted count of the number of
$2$-Hurwitz graphs of type $(g,n,d) =(0,2,4)$ with an
ordered partition $4 = 3+1$.

\begin{figure}[htb]
\label{fig:H02}
\includegraphics[width=2in]{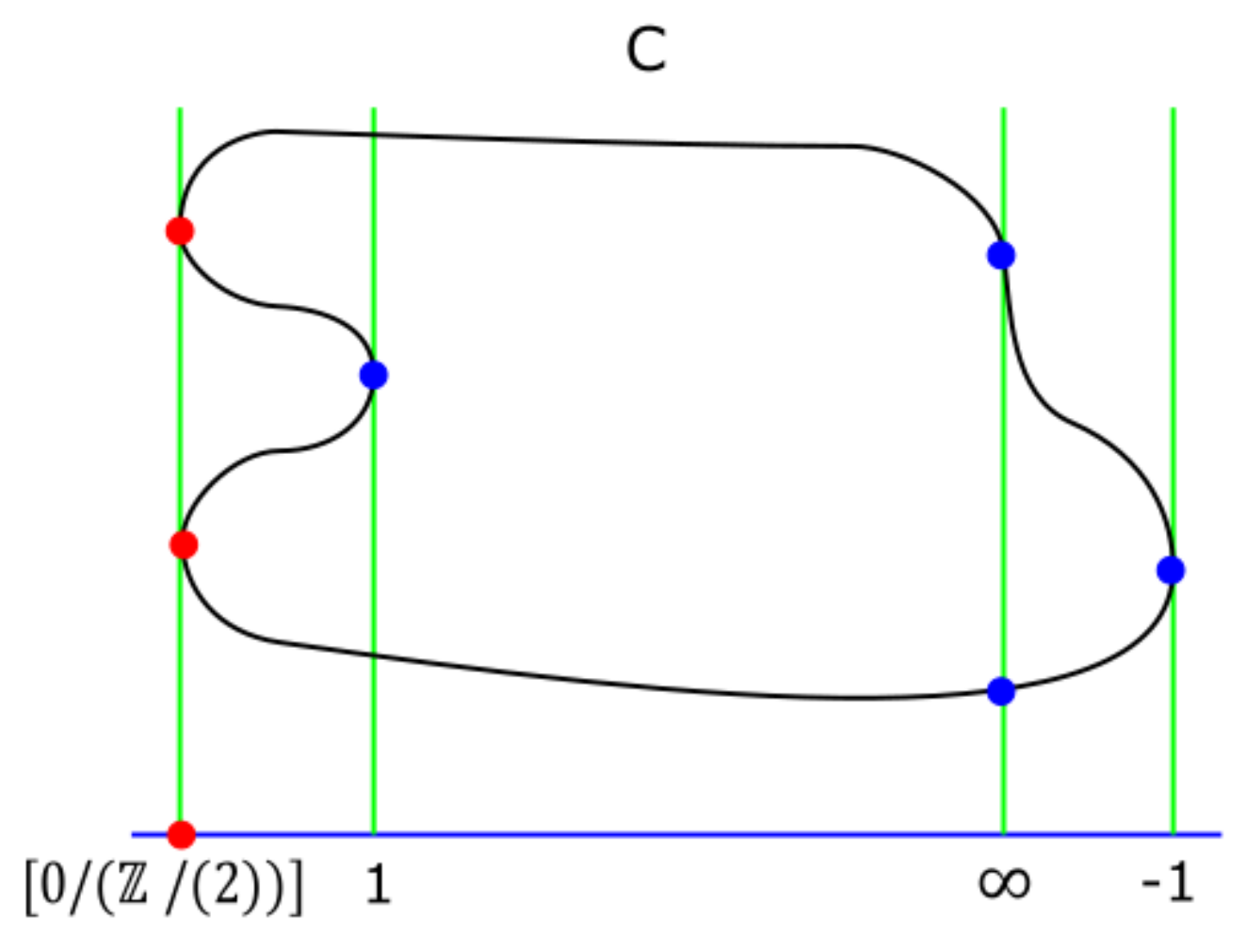}
\caption{Hurwitz covers counted in 
$\cH_{0,2}^2(3,1)$ have two orbifolds points,
two simple ramification points, and one 
ramification point of degree $3$.
}
\end{figure}

\begin{figure}[htb]
\includegraphics[height=0.6in]{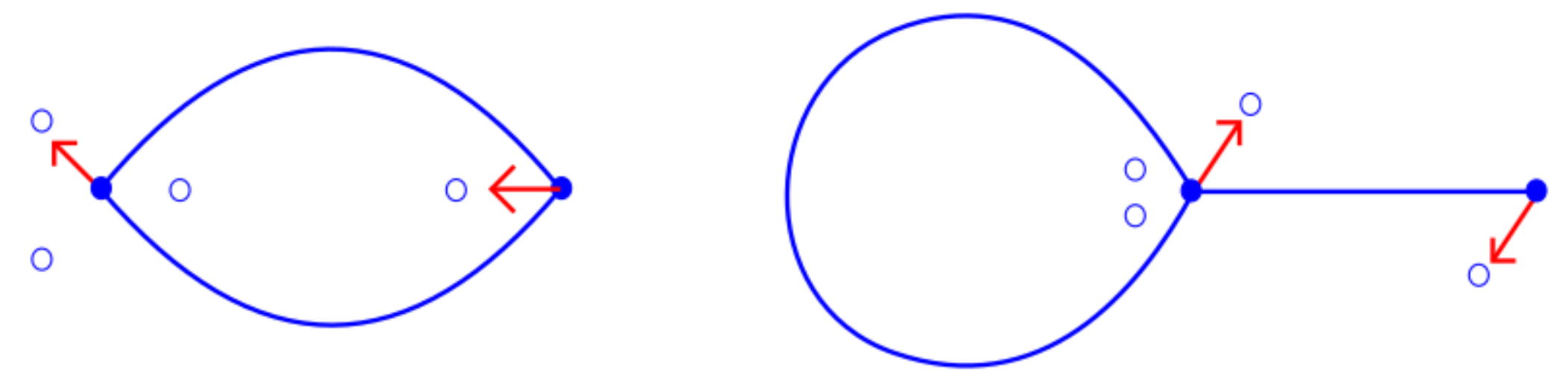}
\caption{There are two $2$-Hurwitz graphs.
The number of graphs is $3/2$ for the graph on 
the left counting the automorphism, 
and $3$ for the one on the right. The total
is thus $9/2$.}
\label{fig:H02graph}
\end{figure}

In terms of formulas, the $2$-Hurwitz cover
corresponding to the graph on the left of 
Figure~\ref{fig:H02graph} is given by
$$
f(x) = \frac{(x-1)^2(x+1)^2}{x}.
$$
To make the simple ramification points
sit on $\pm 1$, we need to divide $f(x)$ by
$f(i/\sqrt{3})$, where $x=\pm 1/\sqrt{3}$ are  
 the simple ramification points. 
 The $2$-Hurwitz cover corresponding to
 the graph on the right of 
 Figure \ref{fig:H02graph} is given by
 $$
f(x) = \frac{(x-1)^2(x+1)^2}{x-a},
$$
where $a$ is a real number
satisfying  $|a|>\sqrt{3}/2$. The real parameter $a$
changes the topological type of the $2$-Hurwitz
cover. For $-\frac{\sqrt{3}}{2}<a<\frac{\sqrt{3}}{2}$,
the graph is the same as on the left, and for
$|a|>\frac{\sqrt{3}}{2}$, the graph becomes
the one on the right.
\end{ex}

\subsection{The edge-contraction formulas}

\begin{Def}[Edge-contraction operations]
The edge-contraction operations (ECOs) on an
arrowed $r$-Hurwitz graph 
$(\gam,\vec{D})$ are the following procedures.
Choose an edge $E$ of the cell graph $\gam$.
\begin{itemize}

\item \textbf{ECO 1}: 
We consider the case that
$E$ is an edge connecting two distinct
vertices Vertex $i$ and Vertex $j$. We can assume
$i<j$, which  induces a direction
$i\overset{E}{\lrar} j$ on $E$. Let us denote
by $F_+$ and $F_-$ the faces bounded
by $E$, where $F_+$ is on the left side of $E$
with respect to the direction. We now contract 
$E$, with the following additional 
operations:
\begin{enumerate}
\item Remove the original arrows at Vertices $i$ and
$j$.
\item Put the dots on $F_\pm$ clustered at
Vertices $i$ and $j$ together, keeping the cyclic
order of the dots on each of $F_\pm$.
\item Place a new arrow to the largest dot
on the corner at Vertex $i$ of Face $F_+$
with respect to the cyclic order.
\item If there are no dots on this particular corner,
then place an arrow to the first 
dot we encounter according to the counter-clock wise 
rotation from $E$ and centered at Vertex $i$.
\end{enumerate}

\end{itemize}
The new arrow at the joined vertex allows us
to recover the original graph from the new one.

\begin{figure}[htb]
\includegraphics[height=0.9in]{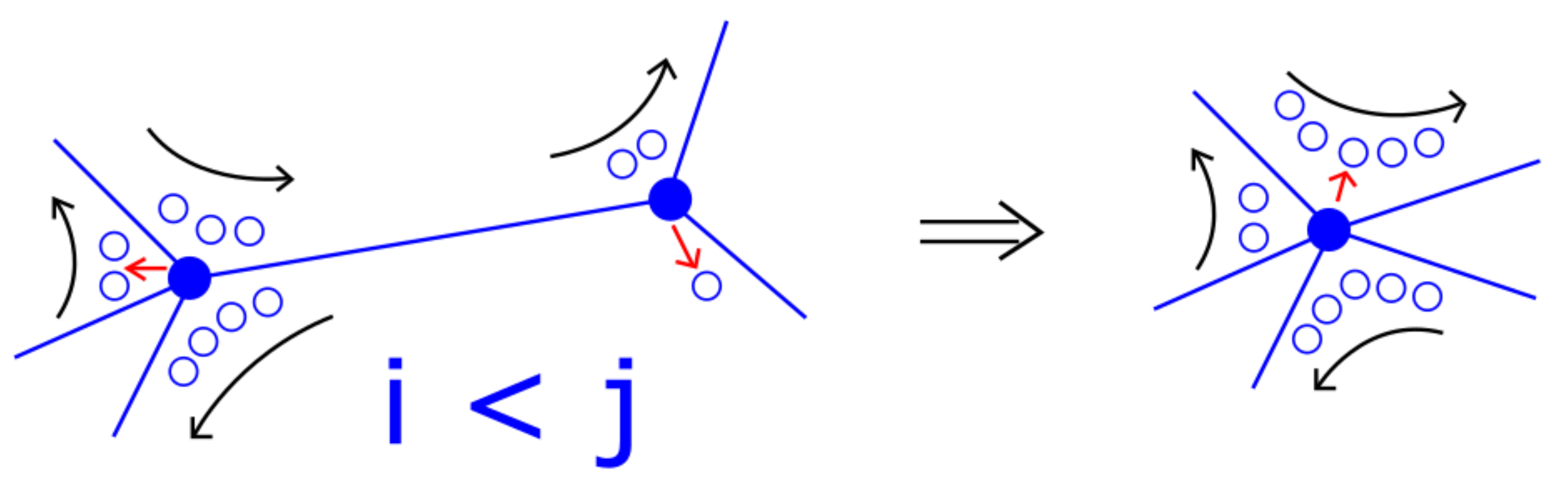}
\caption{After 
contracting the edge, a new arrow is placed
on the dot that is the largest
(according to the cyclic order) around Vertex $i$
in the original graph,
and on the face incident to $E$ which is 
on the left of $E$ with respect to the 
direction $i\rar j$. The new arrow tells us 
where the break is made in the original 
graph.
If there are no dots on this particular face,
then we go around Vertex $i$ counter-clock
wise and find the first dot in the original graph.
We place an arrow to this dot in the new
graph after contracting $E$. Here again the
purpose is to identify which of the $\mu_i$ dots
come from the original Vertex $i$
}
\label{fig:ECO1}
\end{figure}

\begin{itemize}

\item \textbf{ECO 2}: This time $E$ is a loop 
incident to Vertex $i$ twice. We contract 
$E$ and separate the vertex into two new ones,
as in ECA 3 of Definition~\ref{def:ECA}.
The additional operations are:
\begin{enumerate}
\item The contraction of a loop does not change
the number of faces. Separate the dots clustered at 
Vertex $i$ according to the original 
configuration.
\item Look at the new vertex to which the
original arrow is placed. We keep the same
name  $i$ to this vertex. The other vertex is named
$i'$.

\item Place a new arrow to the  dot 
on the corner at the new Vertex $i$
that was the largest in the original corner
with respect to the cyclic order.
\item If there are no dots on this particular corner,
then place an arrow to the first 
dot we encounter according to the counter-clock wise 
rotation from $E$ and centered at Vertex $i$
on the side of the old arrow.

\item We do the same operation for the new
Vertex $i'$, and put a new arrow to a dot.

\item Now remove the original arrow.
\end{enumerate}

\end{itemize}
\begin{figure}[htb]
\includegraphics[height=0.9in]{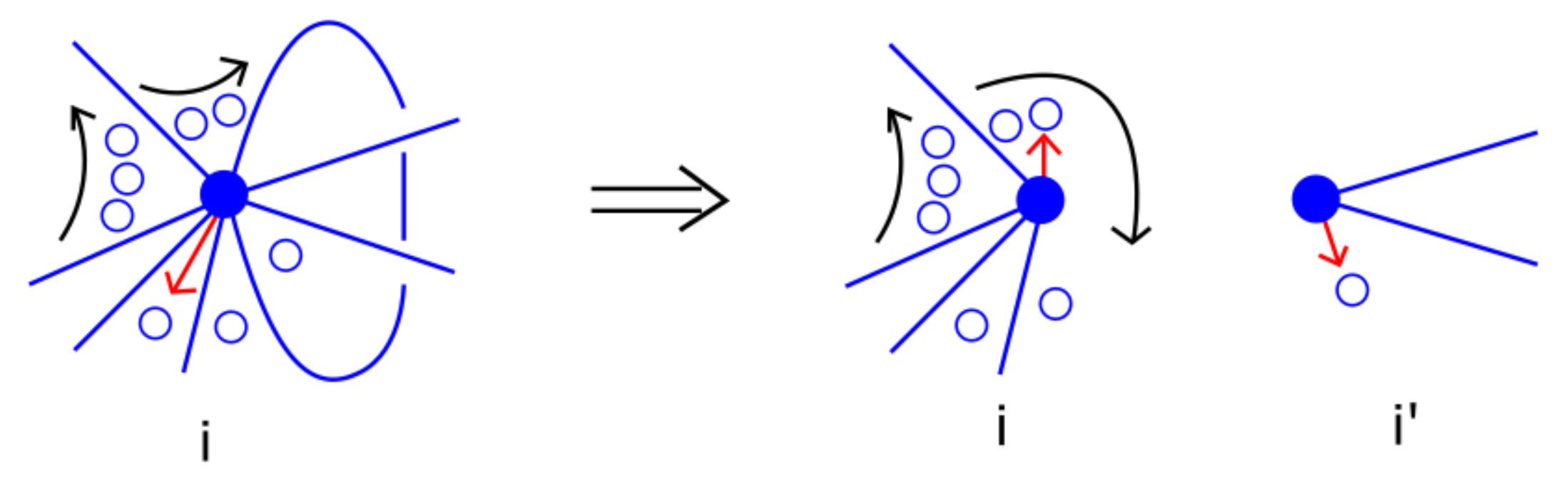}
\caption{New arrows are placed so that
the original graph can be recovered from the
new one
}
\label{fig:ECO2}
\end{figure}

\end{Def}

Although cumbersome, it is easy to show that

\begin{lem}
The edge-contraction operations preserve
the set of $r$-Hurwitz graphs.
\end{lem}

An application of the edge-contraction operations
is the following counting recursion formula.

\begin{thm}[Edge-Contraction Formula]
\label{thm:ECF}
The number of arrowed Hurwitz graphs
satisfy the following edge-contraction formula.
\be
\label{ECF}
\begin{aligned}
 &\left(2g-2+\frac{d}{r}+ n\right)
 \cH_{g,n}^r(\mu_1\dots,\mu_n) 
\\
&\qquad
=
\sum_{i< j}\mu_i\mu_j
\cH_{g,n-1}^r (\mu_1,\dots,\mu_{i-1},
\mu_i+\mu_j,\mu_{i+1},\dots,
\widehat{\mu_j},\dots,\mu_n)
\\
&\qquad+
\half \sum_{i=1}^n
\mu_i \sum_{\substack{\a+\b=\mu_i\\
\a, \b\ge 1}}
\left[\cH_{g-1,n+1}^r (\a,\b,\mu_1,\dots,
\widehat{\mu_i},\dots,\mu_n)
\phantom{\sum_{\substack{g_1+g_2=g\\
I\sqcup J = \{1,\dots,\hat{i},\dots,n\}}}}
\right.
\\
&\qquad+
\left.
\sum_{\substack{g_1+g_2=g\\
I\sqcup J = \{1,\dots,\hat{i},\dots,n\}}}
\cH_{g_1,|I|+1}^r (\a,\mu_I)
\cH_{g_2,|J|+1}^r (\b,\mu_J)
\right].
\end{aligned}
\ee
Here, $\widehat{\;\;}$ indicates the omission
of the index, and $\mu_I = (\mu_i)_{i\in I}$
for any subset $I\subset\{1,2,\dots,n\}$.
\end{thm}
\begin{rem}
The edge-contraction formula (ECF) is a
recursion with respect to the number of edges
$$
s = 2g-2+\frac{\mu_1+\cdots+\mu_n}{r} +n.
$$
Therefore, it  calculates all values  of 
$\cH_{g,n}^r(\mu_1\dots,\mu_n)$
from the base case $\cH_{0,1}^r(r)$.
However,
it does not determine the initial value itself,
since $s=0$. We also note that the recursion 
is not for
$\cH_{g,n}^r$ as a function in $n$ integer
variables.
\end{rem}

\begin{proof}
The counting is done by applying the edge-contraction
operations. The left-hand side of
\eqref{ECF} shows the choice
of an edge, say $E$, out of 
$
s = 2g-2+\frac{d}{r}+n
$
 edges.
The first line of the right-hand side
corresponds to the case that the chosen edge $E$
connects Vertex $i$ and Vertex $j$. 
We  assume $i<j$, and apply ECO~1. The factor 
$\mu_i\mu_j$ indicates the removal of 
two arrows at these vertices
(Figure~\ref{fig:ECO1}).

When the edge $E$ we have chosen is a loop 
incident to Vertex $i$ twice, then we
apply ECO~2. The factor $\mu_i$ is the removal 
of the original arrow (Figure~\ref{fig:ECO2}).
The second and third lines on the right-hand side
correspond whether $E$ is a handle-cutting
loop, or a separation loop.
The factor $\half$ is there because of the 
symmetry between $\a$ and $\b$ of the partition
of $\mu_i$.
This complete the proof.
\end{proof}

\begin{thm}[Graph enumeration and orbifold
Hurwitz numbers]
The graph enumeration and counting orbifold
Hurwitz number are related by the following 
formula:
\be
\label{equivalence}
\cH_{g,n}^r(\mu_i,\dots,\mu_n)
= \mu_1\mu_2\cdots\mu_n 
H_{g,n}^r(\mu_i,\dots,\mu_n).
\ee
\end{thm}

\begin{proof}
The simplest orbifold Hurwitz number is
$H_{0,1}^r(r)$, which counts double
Hurwitz numbers with the same profile
$(r)$ at both $0\in \bP^1$ and 
$\infty\in \bP^1$. There is only
one such map $f:\bP^1\lrar \bP^1$,
which is given by $f(x) = x^r$. 
Since the map has automorphism 
$\bZ/(r)$, we have $H_{0,1}^r(r) = 1/r$. 
Thus \eqref{equivalence} holds for the base 
case. 

We notice that \eqref{ECF} is exactly the
same as the cut-and-join equation of
\cite[Theorem~2.2]{BHLM},
after modifying the orbifold Hurwitz numbers
by multiplying $\mu_1\cdots\mu_n$. Since
the initial value is the same,
and the formulas are recursion based on 
$s=2g-2+\frac{d}{r}+ n$,  
\eqref{equivalence} holds by induction.
This completes the proof.
\end{proof}

\section{Construction of the mirror spectral curves
for  orbifold Hurwitz numbers}
\label{sect:spectral}

In the earlier work on simple and
orbifold Hurwitz numbers
in connection to the topological recursion
\cite{BHLM,BM,DLN,EMS,MZ}, the spectral curves
are determined by the infinite framing limit of the 
mirror curves to  toric Calabi-Yau (orbi-)threefolds.
The other ingredients of the topological recursion,
the differential forms $W_{0,1}$ and
$W_{0,2}$, are calculated by the Laplace
transform of the $(g,n)=(0,1)$ and 
$(0,2)$ cases of the ELSV \cite{ELSV} and
JPT \cite{JPT} formulas. Certainly the logic is
clear, but why these choices are the right ones
is not well explained.

In this section, we show that  the 
edge-contraction operations themselves determine
all the mirror ingredients, i.e.,
the spectral curve, $W_{0,1}$, and $W_{0,2}$.
The structure of the story is the following.
The edge-contraction formula \eqref{ECF}
is an equation among different values of $(g,n)$.
When restricted to $(g,n)= (0,1)$, it produces an
equation on $\cH_{0,1}^r(d)$ as a function
in one integer variable. The generating function
of $\cH_{g,n}^r(\mu_1,\dots,\mu_n)$ is 
reasonably complicated, but it can be
expressed rather nicely in terms of the
generating function of the $(0,1)$-values
$\cH_{0,1}^r(d)$, which is essentially
the spectral curve of the theory.
The edge-contraction formula \eqref{ECF}
itself has the Laplace transform that can
be calculated in the spectral curve coordinate.
Since \eqref{ECF} contains $(g,n)$ on each
side of the equation, to make it a genuine
recursion formula for functions with respect
to $2g-2+n$ in the stable range, we need
to calculate the generating functions of
$\cH_{0,1}^r(d)$ and $\cH_{0,2}^r(\mu_1,\mu_2)$,
and make the rest of \eqref{ECF} free of unstable
terms. The result is the topological recursion
of \cite{BHLM,EMS}.

Let us now start with the restricted \eqref{ECF} 
on $(0,1)$ invariants:
\be
\label{ECF01}
\left(\frac{d}{r} -1
\right)\cH_{0,1}^r (d)
=\half  d \sum_{\substack{\a+\b=d\\\a,\b\ge 1}}
\cH_{0,1}^r (\a)\cH_{0,1}^r (\b).
\ee
At this stage, we introduce a generating function
\be
\label{H01generating}
y=y(x)=\sum_{d=1}^\infty \cH_{0,1}^r (d) x^d.
\ee
In terms of this generating function, \eqref{ECF01}
is  a differential equation
\be
\label{ECFdiff}
\left(x^{r+1} \circ\frac{d}{dx}\circ \frac{1}{x^r}\right)
y = \half r x \frac{d}{dx} y^2,
\ee
or simply
$$
\frac{y'}{y}-ry'=\frac{r}{x}.
$$
Its unique solution is
$$
Cx^r=  y e^{-ry}
$$
with a constant of integration $C$.
As we noted in the previous section,
the recursion \eqref{ECF} does not determine
the initial value $\cH_{0,1}^r(d)$. 
For our graph enumeration problem, the values 
are
\be
\label{H01initial}
\cH_{0,1}^r(d) = 
\begin{cases}
0 \qquad 1\le d<r;
\\
1 \qquad d = r,
\end{cases}
\ee
which determine $C=1$.
Thus we find
\be
\label{r-Lambert}
x^r = y e^{-ry},
\ee
which is the $r$-Lambert curve of \cite{BHLM}.
This is indeed the spectral curve for the
orbifold Hurwitz numbers.

\begin{rem}
We note that
$r\cH_{0,1}^r(rm)$ satisfies the same
recursion equation \eqref{ECF01} for $r=1$,
with a different initial value. Thus essentially
orbifold Hurwitz numbers are determined by
the usual simple Hurwitz numbers.
\end{rem}

\begin{rem}
If we define $T_d = (d-1)! \cH_{0,1}^{r=1}(d)$,
then \eqref{ECF01} for $r=1$ is equivalent to
\eqref{tree}. This is the reason we consider the
tree recursion as the spectral curve for simple
and orbifold Hurwitz numbers.
\end{rem}

For the purpose of performing  analysis on
the spectral curve
\eqref{r-Lambert}, let us introduce a global
coordinate $z$ on the $r$-Lambert curve,
which is an analytic curve of genus $0$:
\be
\label{z}
\begin{cases}
x = x(z) := z e^{-z^r}\\
y =y(z) :=  z^r.
\end{cases}
\ee
We denote by $\Sigma\subset \bC^2$ this parametric
curve.
Let us introduce the generating
functions of  general $\cH_{g,n}^r$, which are
called \emph{free energies}:
\be
\label{FgnHur}
F_{g,n}(x_1,\dots,x_n):=
\sum_{\mu_1,\dots,\mu_n\ge 1}
\frac{1}{\mu_1\cdots \mu_n}
\cH_{g,n}^r (\mu_1,\dots,\mu_n)\prod_{i=1}^n
x_i ^{\mu_i}.
\ee
We also define the exterior derivative
\be
\label{WgnHur}
W_{g,n}(x_1,\dots,x_n):= d_1\cdots d_n F_{g,n}
(x_1,\dots,x_n),
\ee
which is a symmetric $n$-linear differential form.
By definition, we have
\be
\label{F01}
y=y(x) = x\frac{d}{dx} F_{0,1}(x).
\ee
The topological recursion requires the spectral
curve, $W_{0,1}$, and $W_{0,2}$. From
\eqref{WgnHur} and 
\eqref{F01}, 
we have
\be
\label{W01Hur}
W_{0,1}(x) = y\frac{dx}{x} = y d\log(x).
\ee

\begin{rem}
For many examples of topological recursion
such as ones considered in 
\cite{DMSS}, we often define
$W_{0,1} = ydx$, which is a holomorphic
$1$-form on the spectral curve. For Hurwitz theory,
due to \eqref{F01}, it is more natural to use
\eqref{W01Hur}.
\end{rem}

As a differential equation, we can solve
\eqref{F01} in a closed formula on the
spectral curve $\Sigma$ of \eqref{z}.
Indeed, the role of the spectral curve is
that the free energies, i.e., $F_{g,n}$'s, are 
actually  analytic functions defined on 
$\Sigma^n$. Although we define
$F_{g,n}$'s as a formal power series
in $(x_1,\dots,x_n)$ as generating functions,
they are analytic, and the domain of analyticity,
or the classical sense of \emph{Riemann surface},
is the spectral curve $\Sigma$.
The  coordinate change \eqref{z} gives us
\be
\label{xd/dx}
x\frac{d}{d x} = \frac{z}{1-rz^r}
\frac{d}{d z},
\ee
hence  \eqref{F01} is equivalent to 
$$
z^{r-1}(1-rz^{r}) = \frac{d}{dz} F_{0,1}
\big(x(z)\big).
$$
Since $z=0\Longrightarrow x=0\Longrightarrow
F_{0,1}(x)=0$, we find
\be
\label{F01(z)}
F_{0,1}\big(x(z)\big) = \frac{1}{r} z^r -\half z^{2r}.
\ee

The calculation of $F_{0,2}$ is done
similarly, by restricting \eqref{ECF} to the
$(g,n)=(0,1)$ and $(0,2)$ terms. Assuming that
$\mu_1+\mu_n = mr$, we have
\begin{multline}
\label{ECF02}
\left(\frac{d}{r}-1\right)\cH_{0,2}^r(\mu_1,\mu_2)
\\
=\mu_1\mu_2 \cH_{0,1}^r (\mu_1+\mu_2)
+ \mu_1 \sum_{\substack{\a+\b=\mu_1\\\a,\b>0}}
\cH_{0,1}^r(\a)\cH_{0,2}^r(\b,\mu_2)
+ \mu_2 \sum_{\substack{\a+\b=\mu_2\\\a,\b>0}}
\cH_{0,1}^r(\a)\cH_{0,2}^r(\mu_1,\b).
\end{multline}
As a special case of \cite[Lemma~4.1]{BHLM},
this equation translates into a differential equation
for $F_{0,2}$:
\begin{multline}
\label{PDEF02x}
\frac{1}{r}\left(x_1\frac{\partial}{\partial x_1}
+ x_2\frac{\partial}{\partial x_2}\right)
F_{0,2}(x_1,x_2)
\\
=
\frac{1}{x_1-x_2}\left(
x_1^2\frac{\partial}{\partial x_1}F_{0,1}(x_1)-
x_2^2\frac{\partial}{\partial x_2}F_{0,1}(x_2)
\right)
-
\left(
x_1\frac{\partial}{\partial x_1}F_{0,1}(x_1)+
x_2\frac{\partial}{\partial x_2}F_{0,1}(x_2)
\right)
\\
+
\left(
x_1\frac{\partial}{\partial x_1}F_{0,1}(x_1)
\right)
\left(
x_1\frac{\partial}{\partial x_1}F_{0,2}(x_1,x_2)
\right)
+
\left(
x_2\frac{\partial}{\partial x_2}F_{0,1}(x_2)
\right)
\left(
x_2\frac{\partial}{\partial x_2}F_{0,2}(x_1,x_2)
\right).
\end{multline}
Denoting by $x_i = x(z_i)$
and using \eqref{xd/dx}, 
\eqref{PDEF02x} becomes simply
\be
\label{PDEF02z}
\frac{1}{r}\left(
z_1\frac{\partial}{\partial z_1}+
z_2\frac{\partial}{\partial z_2}
\right)
F_{0,2}\big(x(z_1),x(z_2)\big)
=
\frac{x_1 z_1^r-x_2 z_2^r}{x_1-x_2}
-(z_1^r+z_2^r)
\ee
on the spectral curve $\Sigma$.
This is a linear partial differential equation
of the first order
with analytic coefficients in the neighborhood
of $(0,0)\in \bC^2$, hence
by the Cauchy-Kovalevskaya theorem,  it has 
the unique analytic 
solution around the origin of $\bC^2$
for any Cauchy problem.
Since the only analytic solution to 
the homogeneous equation
$$
\left(
z_1\frac{\partial}{\partial z_1}+
z_2\frac{\partial}{\partial z_2}
\right)
f(z_1,z_2)=0
$$
is a constant, the initial condition
$F_{0,2}(0,x_2) = F_{0,2}(x_1,0) = 0$ determines
the unique solution of \eqref{PDEF02z}.

\begin{prop}
We have a closed formula for $F_{0,2}$
in the $z$-coordinates:
\be
\label{F02z}
F_{0,2}\big(x(z_1),x(z_2)\big) = \log\frac{z_1-z_2}{x(z_1)-x(z_2)}
-(z_1^r+z_2^r).
\ee
\end{prop}

\begin{proof}
We first note  that 
$\log\frac{z_1-z_2}{x(z_1)-x(z_2)}$ is holomorphic
around $(0,0)\in \bC^2$. 
\eqref{F02z} being a solution to \eqref{PDEF02z}
is a straightforward calculation that can be verified 
as follows:
\begin{align*}
&\left(
z_1\frac{\partial}{\partial z_1}+
z_2\frac{\partial}{\partial z_2}
\right)
\log\frac{z_1-z_2}{x(z_1)-x(z_2)}
\\
&=
\frac{z_1-z_2}{z_1-z_2}
-\frac{z_1e^{-z_1^r}(1-rz_1^r)-
z_2e^{-z_2^r}(1-rz_2^r)}{x_1-x_2}
\\
&=
1
-\frac{x_1-x_2}{x_1-x_2}
+r\frac{x_1z_1^r-x_2z_2^r}{x_1-x_2}
=r\frac{x_1z_1^r-x_2z_2^r}{x_1-x_2}.
\end{align*}
Since
$
F_{0,2}\big(x(0),x(z_2)\big) 
=\log  e^{z_2^r} - z_2^r = 0,
$
\eqref{F02z} is the desired unique solution.
\end{proof}

In \cite{BHLM}, the functions  
\eqref{F01(z)}
and
\eqref{F02z} are derived
by directly computing the Laplace transform
of the JPT formulas \cite{JPT}
\be
\label{JPT}
\begin{aligned}
&H_{0,1}^r(d) = \frac{d^{\lfloor \frac{d}{r}\rfloor-2}}
{\lfloor \frac{d}{r}\rfloor!},
\\
&H_{0,2}^r (\mu_1,\mu_2) = 
\begin{cases}
r^{\la \frac{\mu_1}{r}\ra +\la \frac{\mu_1}{r}
\ra}
\frac{1}{\mu_1+\mu_2}
\frac{\mu_1^{\lfloor \frac{\mu_1}{r}\rfloor}
\mu_2^{\lfloor \frac{\mu_2}{r}\rfloor}
}
{\lfloor \frac{\mu_1}{r}\rfloor!
\lfloor \frac{\mu_2}{r}\rfloor!}
\qquad \mu_1+\mu_2 \equiv 0 \mod r
\\
0\hskip1.96in \text{otherwise}.
\end{cases}
\end{aligned}
\ee
Here, 
$q =  \lfloor q\rfloor
+\la q\ra$ gives the decomposition of
a rational number $q\in \bQ$ into its
floor  and the fractional part.
We have thus recovered  \eqref{JPT}
from the
edge-contraction formula alone, which are
the $(0,1)$ and $(0,2)$ cases
 of the ELSV formula
for the orbifold Hurwitz numbers.

\begin{ack}
The paper is based a series of lectures
by M.M.\ at Mathematische Arbeitstagung 2015, 
Max-Planck-Institut f\"ur Mathematik in Bonn.
The authors  are grateful to  
the American Institute of Mathematics in California, 
the Banff International Research Station,
 the Institute for Mathematical
Sciences at the National University of Singapore,
 Kobe University, Leibniz Universit\"at 
 Hannover, the Lorentz Center for Mathematical 
 Sciences, Leiden, 
 Max-Planck-Institut f\"ur Mathematik in Bonn, 
 and Institut Henri Poincar\'e, Paris,
for their hospitality and financial support during
the authors' stay for collaboration. 
 The research of O.D.\ has been supported by
 GRK 1463 \emph{Analysis,
Geometry, and String Theory} at 
Leibniz Universit\"at 
 Hannover and MPIM.
The research of M.M.\ has been supported 
by 
NSF grants DMS-1309298, 
DMS-1619760, DMS-1642515, 
and NSF-RNMS: Geometric Structures And 
Representation Varieties (GEAR Network, 
DMS-1107452, 1107263, 1107367).
\end{ack}


\providecommand{\bysame}{\leavevmode\hbox to3em{\hrulefill}\thinspace}

\bibliographystyle{amsplain}

\end{document}